\documentclass[final,smallcondensed]{svjour3}
 \usepackage{ifthen}\newboolean{1014} \setboolean{1014}{true}
\usepackage{amsmath,amssymb}

\def\Ga{{\Gamma}}

\def\de{\delta}
\def\al{\alpha}

\def\ga{\gamma}

\def\Sym{{\rm Sym}\,}

\def\Ree{{\rm Ree}\,}

\def\PGaL{{\rm P\Gamma L}\,} 
\def\Sp{{\rm Sp}\,}

\def\PG{\rm PG}

\def\Aut{{\rm Aut}\,}

\def\ov{\overline}

\def\V{\mathcal{V}}

\def\w{\mathsf{w}}
\def\u{\mathsf{u}}

\title{Sporadic neighbour-transitive codes in Johnson graphs\thanks{Written in honour of Frank de Clerck}} 
\author{Max Neunh{\"o}ffer \and Cheryl E. Praeger}
\date{7~June 2013}
\institute{Max Neunh{\"o}ffer \at
School of Mathematics and Statistics,
Mathematical Institute,
North Haugh,
St Andrews, Fife KY16 9SS,
Scotland, United Kingdom, \email{neunhoef@mcs.st-and.ac.uk}
\and Cheryl E. Praeger \at
School of Mathematics and Statistics,
The University of Western Australia (M019),
35 Stirling Highway,
Crawley, WA 6009,
Australia; also affiliated with King 
Abdulaziz University, Jeddah, Saudi Arabia \email{cheryl.praeger@uwa.edu.au}
The second author was supported by Australian Research Council Federation Fellowship FF0776186.}
\journalname{Designs, Codes and Cryptography}

\begin{document}
\maketitle

\begin{abstract}
We classify the neighbour-transitive codes in John\-son graphs 
$J(v,k)$ of minimum distance at least $3$ which admit a neighbour-transitive
group of automorphisms that is an almost simple $2$-transitive 
group of degree $v$ and does not occur in an infinite family 
of $2$-transitive groups.  The result of 
this classification is a table of $22$ codes with these properties. 
Many have relatively large minimum distance in comparison to their 
length $v$  and number of code words. We construct an additional 
$5$ neighbour-transitive codes  with minimum 
distance $2$ admitting such a group. All $27$ codes are $t$-designs with $t$ at least $2$.
\end{abstract}


\section{Introduction}\label{intro}

This paper is a contribution to the study of error-correcting codes in 
the Johnson graphs such that all codewords are equivalent and also all 
code-neighbours are equivalent under symmetries of the code, that is to say, 
the study of \emph{neighbour-transitive codes}. Our approach is in the spirit of 
Delsarte's program \cite{Delsarte} to investigate completely regular codes in 
distance-regular graphs, and is a response to the disappointingly small 
numbers of such codes found over the years 
with good error-correcting properties (large minimum distance).
Delsarte \cite{Delsarte} in particular had asked about 
the existence of completely regular codes in Johnson graphs, and such 
codes have been studied by Meyerowitz~\cite{Meyerowitz1,Meyerowitz2} 
and Martin~\cite{Martin94,Martin98}. On the one hand, the neighbour-transitivity condition 
relaxes the stringent regularity conditions imposed for complete regularity, 
replacing them with conditions involving only  
codewords and their immediate neighbours. On the other hand the
regularity conditions for codewords and their neighbours are strengthened to
a local transitivity property.
 
The  \emph{Johnson graph} $J(v,k)$, based on a set $\V$ of $v$ elements 
called \emph{points}, where $2\leq k\leq v-2$, is  the graph whose vertex 
set is the set $\binom{\V}{k}$ of all $k$-subsets of $\V$, with 
edges being the unordered pairs $\{\ga,\ga'\}$ of $k$-subsets
such that $|\ga\cap\ga'|=k-1$.
The graph $J(v,k)$ admits the symmetric group $\Sym(\V)$ 
as a group of automorphisms, and if $k\ne v/2$ this is the full automorphism group. 
If $k=v/2$ the automorphism group $\Sym(\V)\times\langle\tau\rangle$ is twice as large, including in particular the complementing involutory map $\tau$ which maps each $k$-subset 
$\ga$ of $\V$ to its complement $\V\setminus\ga$. 

A code in $J(v,k)$ is a subset $\Gamma$ of the vertex set  $\binom{\V}{k}$,
and its automorphism group $A$ is the set-wise stabiliser of $\Gamma$ in 
$\Aut(J(v,k))$. Code-neighbours of $\Gamma$ are the vertices $\ga_1\not\in\Gamma$ 
that are joined by an edge to at least one codeword $\ga\in\Gamma$; $\Gamma$ is 
said to be \emph{neighbour-transitive} if $A$ is transitive on both $\Gamma$ and 
the set $\Gamma_1$ of code-neighbours; and more generally 
$\Gamma$ is called \emph{$G$-neighbour-transitive}, where $G\leq A$, if $G$ 
is transitive on both $\Gamma$ and $\Gamma_1$. If $k=v/2$ it is possible that 
$A\not\leq \Sym(\V)$ and that a code $\Ga$ is $A$-neighbour-transitive while the 
group $A\cap\Sym(\V)$ is not neighbour-transitive on $\Gamma$; this situation
will be addressed in \cite{NP}. We are concerned in this paper with the case
in which $\Ga$ is $(A\cap\Sym(\V))$-neighbour-transitive.

%
Neighbour-transitive codes $\Gamma$ contained in a Johnson graph
$J(v,k)$ were first studied by Liebler and the second author in
\cite{LP}. All such codes for which the group $G:=A\cap\Sym(\V)$ does
not act primitively on the underlying set $\V$ were explicitly described
in \cite[Theorem 1.1]{LP}.  There are two infinite families of examples for which 
$G$ is intransitive in its action on $\V$. If $G$ is transitive but 
imprimitive on $\V$, then the classification yields five infinite families
of examples together with a recursive construction of such codes.

The minimal distance of a code $\Gamma$ is the smallest distance  $\delta(\Gamma)$ in $J(v,k)$
between distinct codewords $\gamma_1, \gamma_2$, that is to say 
$\delta(\Gamma)$ is the smallest value attained by $k-|\gamma_1\cap\gamma_2|$.
In the case where $G$ is primitive on $\V$, the analysis in \cite{LP}
focuses on codes $\Gamma$ with $\delta(\Gamma)\geq2$. (see, for example, \cite[Chapter
7.4]{Cam})
It was shown in \cite[Theorem 1.2]{LP} that a $G$-neighbour-transitive
code $\Ga$ with $\delta(\Ga)\geq3$ has the following property, called
\emph{$G$-strong incidence transitivity}: the group $G$ is transitive
on $\Ga$ and, for $\ga\in\Ga$, $G_\ga$ is transitive on the set of
pairs $(\u,\u')$ with $\u\in\ga, \u'\in \V\setminus\ga$. In the case where
$\delta(\Ga)=2$, the same theorem shows that $G$-strong incidence transitivity
is equivalent to $G$-transitivity on pairs $(\gamma,\gamma_1)$ with $\gamma\in\Ga, 
\gamma_1\in\Ga_1$, a property strictly stronger 
than $G$-neighbour transitivity, see \cite[Remark 1.5]{LP}. The major
signifiance of  \cite[Theorem 1.2]{LP} for this paper, however, is its final assertion: 
namely that, if $G$ is primitive on $\V$, then $G$-strong incidence transitivity
implies that $G$ is $2$-transitive on $\V$.  Since the finite $2$-transitive 
permutation groups are known explicitly as a consequence of the classification of 
the finite simple groups (see, for example, \cite{Cam}), this result offers a way forward to a possible 
classification of the $G$-strongly incidence transitive codes in $J(v,k)$.
Such groups $G$ are either of affine type with an elementary abelian normal subgroup
acting regularly on $\V$, or  almost simple, that is  $T\leq G\leq \Aut(T)$ for some finite nonabelian simple group $T$.

In this paper we deal with the cases in which $G$ is a \emph{sporadic almost simple
$2$-transitive group} on $\V$ in the sense that $G$ does not lie in an 
infinite family of almost simple $2$-transitive groups. In Subsection~\ref{sub:summary}
we give a summary of progress on the classification  of $G$-strongly incidence 
transitive codes in $J(v,k)$ for the other types of $2$-transitive permutation groups.
By, for example, \cite[Chapter
7.4]{Cam}, the sporadic almost simple $2$-transitive groups $G$ of degree $v$ are the Mathieu groups $M_v$
for $v\in\{11,12,22,23,24\}$ and $\Aut({\rm M}_{22})$ with $v=22$; ${\rm M}_{11}$ with $v=12$;
$L_2(11)$ with $v=11$; $A_7$ with $v=15$; the Higman-Sims group ${\rm HS}$ with
$v=176$; and Conway's third group ${\rm Co}_3$ with $v=276$. As our main result Theorem~\ref{mathieu}
shows, each of these groups provides at least one neighbour-transitive code
$\Gamma$. Note that $\PGaL(2,8) \cong L_2(8).3 \cong \Ree(3)$ is a bit
ambiguous with respect to the notion ``sporadic almost simple
$2$-transitive group''. As Ree group it is a member of an infinite family,
but it is exceptional in various ways, for example it is the only
$2$-transitive almost simple group whose socle is not $2$-transitive.
However, as explained in Section~\ref{PSL28} it does not provide a
neighbour-transitive code anyway.


 \begin{center}
\begin{table}
\hspace*{-6mm}\begin{tabular}{rlcccrrll}\hline
Line & $G$ &$v$  &$k$& $\delta(\Gamma)$&$|\Ga|$&$A_2$&$\Ga$ & $G_\ga$   \\ \hline 
1& $L_2(11) $ &$11$ &$5$& $3$&11&12& $2$-$(11,5,2)$ biplane & $A_5$   \\ 
2& $A_{7}$ &$15$&$7$&  $4$     &15&16&planes of $\PG(3,2)$& $L_2(7)$\\
3&${\rm M}_{11}$ &$12$ &$6$&$3$&22&24&totals&$A_6$\\
4&${\rm M}_{22}$ &$22$ &$6$&$4$&77&1024&$3$-$(22,6,1)$ design &$2^4:A_6$ \\
5&                  & &$7$&$4$  &176&1024&heptads   &$A_7$\\
6&                &     &$8$&$4$   &330&1024&octads &$2^3:L_3(2)$\\
7&               &       &  $10$&$4$   &616&1024&decads
                 &${\rm M}_{10}\cong A_6\cdot 2_3$\\
8&${\rm M}_{22}.2$ &$22$ &$6$&$4$&77&1024&$3$-$(22,6,1)$ design &$2^4:S_6$ \\
9&                    & &$7$&$3$  &352&6941?&heptads   &$A_7$\\
10&                  &     &$8$&$4$   &330&1024&octads &$2\times 2^3:L_3(2)$\\
11&                 &       &  $10$&$4$   &616&1024&decads&$A_6\cdot(2^2)$\\
12&${\rm M}_{23}$&$23$ &$7$& $4$     &253&2048&$4$-$(23,7,1)$ design &$2^4:A_7$ \\
13&                   &     &$8$&$4$ &506&2048&octads&$A_8$\\
14&            &     &$11$& $4$      &1288&2048&endecads&${\rm M}_{11}$\\
15&${\rm M}_{24}$&$24$ &$8$& $4$  &759&4096&$5$-$(24,8,1)$ design &$2^4:A_8$ \\
16&                &     &$12$&$4$&2576&4096&duum&${\rm M}_{12}$\\
17&${\rm HS}$&$176$ &$50$& $36$   &176&?&$2$-$(176,50,14)$&${\rm U}_3(5):2$  \\
18&          &$   $ &$56$& $32$   &1100&?&$2$-$(176,56,110)$
                   &$L_3(4).2$  \\
19&${\rm Co}_3$&$276$ &$6$& $3$ &708400&?&$2$-$(276,6,280)$
                &$3_+^{1+4}:4S_6$  \\
20&		&     &$36$& $24$&170775&?&$2$-$(276,36,2835)$
                &$2^{.} {\rm Sp}_6(2)$  \\
21& 		&     &$100$&$50$&11178&?&$2$-$(276,100,1458)$ 
                &${\rm HS}$  \\
22&		&     &$126$&$36$ &655776&?&$2$-$(276,126,136080)$
                &${\rm U}_3(5):S_3$  \\
\hline
23& $A_{7}$ &$15$ &$3$& $2$     &35&1024&lines of $\PG(3,2)$ & $(A_3\times A_4).2$ \\ 
24&${\rm M}_{11}$ &$11$ &$5$&$2$&66&72&$4$-$(11,5,1)$ design  & $S_5$ \\ 
25&${\rm M}_{11}$ &$12$ &$6$&$2$&110&144&halves of quadrisect.&$3^2:Q_8$\\
26&${\rm M}_{12}$ &$12$ &$6$&$2$&132&144&$5$-$(12,6,1)$ design& $A_6.2$ \\ 
27&${\rm M}_{24}$  &$24$  &$12$&$2$&35420&344308?&$5$-$(24,12,660)$ &
   $2^6:3.(S_3 \times S_3)$\\
 \hline
&&&
 \end{tabular}\hspace*{-6mm}
\caption{Sporadic 2-transitive Examples, see also \cite{NTCodes}.}
\label{tbl-sporad}
 \end{table}
 \end{center}

\begin{theorem}\label{mathieu} 
Let $G$ be one of the sporadic $2$-transitive groups on a set $\V$
of size $v$ as above, let $2\leq k\leq\frac{|\V|}{2}$,
and suppose that $\Ga\subset\binom{\V}{k}$ is $G$-neighbour-transitive with $\delta(\Ga)\geq3$.
Then $G, v, k, \delta(\Gamma)$, and $\ga\in\Ga$ are as in one of the Lines
$1$--$22$ of Table~{\rm\ref{tbl-sporad}} (above the horizontal line). 
Moreover the codes in lines $3$ and $16$ are self-complementary 
and their full automorphism group is $G \times \left<\tau\right>$.
\end{theorem}

The examples are all linked to interesting geometrical or combinatorial 
configurations, and in each case it can be helpful to view $\Ga$ as the block 
set of a design based on $\V$. As mentioned in
Section~\ref{HammingSect}, our codes can be interpreted as non-linear
binary codes. 
We compared their minimum distance, 
with the known bounds
for the Hamming minimum distance for binary codes from
\cite{NLCodes} (for non-linear codes) and \cite{CodeTables} (for
linear codes). Nearly all of our examples have rather 
large minimum distance. We have added the upper bound from
\cite{NLCodes} (if known) in the column labelled with $A_2$ in
Table~{\rm\ref{tbl-sporad}}.
Thus, for example the code in Line~3 has length
$12$, Hamming minimum distance $6=2\delta(\Gamma)$ and contains $22$
code words which is very close to the upper bound of $24$ for such
codes. The code in Line~7 has length $22$, Hamming minimum
distance $8$ and contains $616$ code words; here the upper bound is
$1024$. The code in Line~16 has length $24$, Hamming minimum distance
$8$ and contains $2576$ codewords (note $2^{11} < 2576 < 2^{12}$). The upper bound
for the number of code words for a non-linear code is $4096=2^{12}$. For linear
binary codes of length $24$ and dimension $11$ or $12$, Hamming
minimum distance $8$ is best possible, so again, our code is very
close to this.

We make some comments on this classification and our approach to proving it.

\subsection{Summary of the concepts}\label{summary}
The codes we study are subsets $\Ga\subseteq\binom{\V}{k}$ of $k$-subsets of $\V$.
The \emph{automorphism group} $\Aut(\Ga)$ of  
$\Ga$ is the set-wise stabiliser of $\Ga$ in the automorphism group $\Aut(J(v,k))$ of $J(v,k)$, and the latter group is $\Sym(\V)$ if $k\ne v/2$ and $\Sym(\V)\times\langle\tau\rangle$ if $k=v/2$ with $\tau$ the complementing map 
which takes each $k$-subset of $\V$ to its complement.
By a \emph{neighbour of $\Ga$} we mean a $k$-subset $\ga_1$ of 
$\V$ that is not a codeword but satisfies $|\ga_1\cap\ga|=k-1$ for some 
codeword $\ga\in\Ga$, that is to say, the distance $d(\ga,\ga_1)$ between
$\ga$ and $\ga_1$ in $J(v,k)$ is 1.
%
Thus provided the minimal distance $\de(\Ga)>1$, 
all vertices adjacent to a codeword are neighbours, and in particular $\Ga$ is a proper 
subset of $\binom{\V}{k}$. For $G\leq\Aut(\Ga)$, we say that $\Ga$ is 
\emph{$G$-neighbour-transitive} if $G$ is transitive on  both
$\Ga$ and the set $\Ga_1$ of neighbours of $\Ga$. As discussed above we will throughout this paper assume that $G\leq\Aut(\Ga)\cap\Sym(\V)$.


For any code $\Ga\subset\binom{\V}{k}$, 
the set of complements $\Ga':=\{\V\setminus\ga\,|\,\ga\in\Ga\}$ is a code in 
$J(v,v-k)$ with neighbour set $\{\V\setminus\ga\,|\,\ga\in\Ga_1\}$. Moreover $\de(\Ga')=\de(\Ga)$, 
and properties such as neighbour-transitivity, or strong incidence-transitivity introduced 
below, hold for $\Ga$ if and only if they hold for $\Ga'$. 
Thus the assumption $k\leq v/2$ is not restrictive at all in our investigation.

\subsection{Strong incidence transitivity}
Recall that a $G$-neighbour-transitive
code $\Ga$, where $G\leq\Sym(\V)$,  is $G$-strongly incidence transitive if $G$ is transitive
on $\Ga$ and, for $\ga\in\Ga$, $G_\ga$ is transitive on the set of
pairs $(\u,\u')$ with $\u\in\ga, \u'\in \V\setminus\ga$. By 
\cite[Theorem 1.2]{LP}, each $G$-strongly incidence transitive code $\Ga$ has $\delta(\Ga)\geq2$, and 
if $\Ga$ is $G$-neighbour-transitive with
$\delta(\Ga)\geq3$ then $\Ga$ is $G$-strongly incidence transitive.
Our approach to proving Theorem~\ref{mathieu} is to
embark on the stronger classification problem of $G$-strongly incidence
transitive codes and, 
having done this, to check if any of the examples
have minimum distance 2. We prove (see
Section~\ref{sect:organisation}):

\begin{proposition}\label{mathieu2} 
Let $G, v, k$ be as in Theorem~\ref{mathieu} 
and suppose that $\Ga\subset\binom{\V}{k}$ is $G$-strongly incidence
transitive with $\delta(\Ga)\geq2$.
Then $G, v, k$, $\delta(\Gamma)$, and $\ga\in\Ga$ are as in one of the lines
of Table~{\rm\ref{tbl-sporad}}. 
Moreover the codes in lines $3$, $16$ and $25$--$27$ are 
self-complementary 
and their full automorphism group is $G \times \left<\tau\right>$.
\end{proposition}


\subsection{Codes in binary Hamming graphs}
\label{HammingSect}
The binary Hamming graph $H(v,2)$ has as vertices the 
ordered $v$-tuples with entries from $\{0,1\}$, and 
edges those pairs of $v$-tuples which agree in all but one entry.
If we write $\V=\{1,2,\dots,v\}$, then each vertex $\ga$ of the 
Johnson graph $J(v,k)$ can be identified with the binary 
$v$-tuple with $i$-entry 1 if and only if $i\in\ga$. 
In this way $J(v,k)$ is identified with the set of
weight $k$ vertices of $H(v,2)$, and each code $\Ga$ in
$J(v,k)$ is identified with a constant weight code in $H(v,2)$. 
Vertices at distance $d$ in $J(v,k)$ correspond to vertices 
in $H(v,2)$ at distance $2d$ so the minimum distance of $\Ga$, 
viewed as a code in $H(v,2)$, is $2\delta(\Ga)$.
Moreover $\Aut(\Ga)$, in its action on entries, is admitted 
by the code in $H(v,2)$, so neighbour-transitive codes in $J(v,k)$ 
yield codes in $H(v,2)$ with groups transitive on codewords. However 
the neighbours of $\Ga$ in $H(v,2)$ have weights $k\pm1$ and we 
usually have no information as to transitivity on code neighbours. 
  
\subsection{A computational approach}
\label{sect:computational}
If we fix a group $G$ and its transitive action on a set of $v$
points, that is, if we are given $G$ as a permutation group, we can
use the following computational approach to find all possible $k$,
$\gamma$, $\Ga = \gamma G$ and thus $G_\gamma$. We use:

\begin{lemma}
\label{thelemma}
Let $G$ be a transitive group on a set $\V$ of size $v$, let $2 \le k
\le \frac{|\V|}{2}$, and suppose that $G$ acts transitively on
$\Ga\subset\binom{\V}{k}$ and $\Ga$ is $G$-strongly incidence-transitive. 
Then there is an $\ell \ge 1$ and
a chain of subgroups
\[ G_\gamma = H_0 < H_1 < \cdots < H_\ell = G \]
such that each $H_i$ is a maximal subgroup in $H_{i+1}$ for $0 \le i <
\ell$, all $H_i$ with $1 \le i \le \ell$ are transitive on $\V$,
$H_0$ has exactly two orbits $\gamma$ and $\V \setminus \gamma$ on $\V$,
and $G_\gamma$ is transitive on $\gamma \times (V \setminus \gamma)$.
\end{lemma}
\begin{proof}
By definition of $G$-strong incidence-transitivity we get that
$G_\gamma = H_0$ has exactly two orbits $\gamma$ and $\V \setminus
\gamma$, is transitive on $\gamma \times (\V \setminus \gamma)$ and
is the set-wise stabiliser of $\gamma$ in $G$. Therefore, in any
maximal chain $G_\gamma = H_0 < H_1 < \cdots < H_\ell = G$ all groups
$H_i$ with $1 \le i \le \ell$ must be transitive on $\V$ because
otherwise they would fix the set $\gamma$.
\end{proof}
This lemma allows to look for the above situation by looking at
subgroups of $G$. We start with a list of representatives of the
conjugacy classes of maximal subgroups of $G$. For each such $H$ on
the list, we compute the $H$-orbits on $\V$. If there are more than $2$
we discard $H$. If
there are exactly two orbits $\gamma$ and $\V \setminus \gamma$ 
with $|\gamma| \le |V \setminus \gamma|$, we check whether or not
$H$ acts transitively on $\gamma \times (V \setminus \gamma)$.
If not, we discard $H$.
Otherwise we enumerate the $H$-orbits on 
$|\gamma|$-subsets, and in so doing we check whether $H$ is the
set-wise stabiliser of $\gamma$. (In the first stage of this 
process $H$ is maximal in $G$ and then $H$ must be the full set-wise 
stabiliser of each of its orbits.) If so, we check $\delta(\Ga)$ and if it is at
least $2$, we have found an interesting $G$-strongly 
incidence-transitive code. On the other hand if $H$ acts transitively on $\V$, 
we append a list of
representatives of the conjugacy classes of maximal subgroups of $H$ to our candidate list
and consider the next subgroup on the list.

Since our groups for the classification are explicitly given as
permutation groups on not too many points, we can either determine 
representatives for the conjugacy classes of maximal subgroups
by explicit computation or by looking them up in the Atlas of Finite
Group Representations (see \cite{WWWAtlas}).

This approach terminates since we are dealing with finite groups and
it will classify all $G$-strongly incidence-transitive codes $\Ga$
because of Lemma~\ref{thelemma}. It is enough to consider one
representative in each $G$-conjugacy class of subgroups since any
$G$-conjugate of $G_\gamma$ will produce an isomorphic code.

We have actually run this algorithm to completion on all of the
sporadic almost-simple 2-transitive groups mentioned above and
Table~\ref{tbl-sporad} contains all the codes found. Thus we have
proved Proposition~\ref{mathieu2} computationally. However, where
possible, we will give a human-readable proof of our classification in
the next section, since the mathematical arguments enhance the
understanding of the beautiful geometric and group-theoretic
structures underlying these codes. In a few cases, however, we will
refer to the computations to finish off the argument. The values for
the minimum distance $\delta(\Ga)$ have all been determined
computationally. We provide input for the \textsf{GAP} computer
algebra system (see \cite{GAP}) to reproduce all codes found
and to verify our computations on the web page \cite{NTCodes}.

\subsection{Summary of progress with classification in the $2$-transitive case.}\label{sub:summary} 
Suppose that $\Ga$ is a $G$-strongly incidence transitive code in $J(v,k)$ with $\delta(\Ga)\geq2$,
and $G\leq \Sym(\V)$ such that $G$ is a 2-transitive permutation group on $\V$, and $G$ is not one of the sporadic $2$-transitive groups treated in this paper. 
We divide such $2$-transitive 
groups into three broad families (see  \cite[Chapter 7.3 and 7.4]{Cam}):

\begin{enumerate}
 \item[(a)] The affine $2$-transitive groups: $G$ is a group of affine
transformations of a finite vector space $V$ and $G$ contains the translation group 
as a normal subgroup acting regularly on $V$ which we can identify with the underlying set $\V$. 
 \item[(b)] The symplectic groups: $G=\Sp(2n,2)$ acts 2-transitively on one
of two families of quadratic forms which polarise to the symplectic form preserved by $G$.
\item[(c)]  All other infinite families of almost simple $2$-transitive groups.
\end{enumerate}

\noindent
The affine 2-transitive groups are analysed in \cite[Section 6]{LP} and it is shown 
in \cite[Propositions 6.1 and 6.6]{LP} that, for a codeword 
$\ga$ viewed as a subset of $\V$, either (i) $\gamma$ is an affine subspace or complement 
of an affine subspace, or (ii) $q\in\{4,16\}$ and either $\V$ is 1-dimensional with 
$\ga$ a Baer subline, or $\V$ has dimension at least 2 and $\ga$ is a subset of   
class $[0,\sqrt{q}, q]_1$ (that is to say, each affine line meets $\ga$ in $0, \sqrt{q}$ or $q$ points).
For the last case, \cite[Example 6.7]{LP} provides an example in 2-dimensions with $q=4$, namely 
the famous 2-transitive hyperoval $H$ with $|\ga|=6$. N. Durante \cite{Dur} classified geometrically
all subsets of affine points  of class $(0,\sqrt{q}, q)_1$ (see Propositions 2.3, 3.6, 
Corollary 2.4, Theorems 3.13, 3.15 of \cite{Dur}) and used this to classify all possible examples 
with the required symmetry properties \cite[Theorem 3.18]{Dur}: for $q=4$ the additional 
possibilities for $\ga$ are cylinders 
with base the 2-transitive hyperoval $H$, or are  unions of two parallel planes; for $q=16$, the 
additional examples $\ga$ are unions of four parallel planes with secant lines meeting each of 
them in a Baer sub-line.

The symplectic groups are not treated in \cite{LP} and are remarked there as being an open case. 
Some recent work is beginning on them by a PhD student of the second author.

All the other infinite families of almost simple $2$-transitive permutation groups $G$ are
considered in \cite[Sections 7 and 8]{LP}. Corresponding to the infinite families of rank 1 Lie type groups
there are two infinite families of strongly incidence transitive codes and one sporadic example 
\cite[Propositions 7.2]{LP}: namely Baer sublines of the projective 
line ${\rm PG}(1,q_0^2)$ for groups with socle ${\rm PSL}(2,q_0^2)$, blocks of the classical unital 
for the $3$-dimensional unitary groups ${\rm PSU}(3,q)$, and bases for the groups 
${\rm PGU}(3,3)$ (with $v=28, k=6$). The only other examples for the third class of groups
come from the projective groups $G$, where ${\rm PSL}(n,q)\leq G\leq {\rm P}\Ga{\rm L}(n,q)$, 
with $n\geq3$ and $v=\frac{q^n-1}{q-1}$ in the natural action on points
of the projective space ${\rm PG}(n-1,q)$. By \cite[Proposition 7.4]{LP}, either (i) the code consists 
of subspaces or their complements of some fixed dimension, or (ii) each codeword is a subset of 
points of  $\PG(n-1,q)$ of class $[0,x,q+1]_1$ (defined as above), where $x=2$, or $q=q_0^2$ and $x=q_0+1$.
In this case, Durante \cite[Theorem 3.2]{Dur} drew together results about subsets of 
$\PG(n-1,q)$ of class $[0,x,q+1]_1$, and showed in \cite[Theorem 3.3]{Dur} that no such subsets, apart from subspaces 
and their complements, have the symmetry property required for strongly incidence-transitive codes.



\section{Proof of Proposition~\ref{mathieu2}}\label{sect:organisation}

Suppose that $\Ga$ is a subset of $\binom{\V}{k}$ with $\delta(\Ga)\geq2$, 
where $2\leq k\leq
\frac{|\V|}{2}=\frac{v}{2}$. Suppose further  that $G\leq\Aut(\Ga)$ is such that
$\Ga$ is $G$-strongly incidence transitive and $G$ is one of the sporadic 
almost simple 2-transitive groups on $\V$ mentioned in Section~\ref{intro}.
For a subset $\al\subseteq\V$ we often write $\ov\al$ for its
complement $\ov\al:=\V\setminus\al$, so that, by definition, $G$ is transitive 
on $\Ga$ and, for $\ga\in\Ga$, $G_\ga$ acts 
transitively on $\ga\times\ov\ga=\{(\u,\w)\,|\,\u\in\ga, \w\in\ov\ga\}$. 
As noted in Subsection~\ref{summary},  $\Ga$ is a proper subset of $\binom{\V}{k}$
since $\delta(\Ga)>1$. It follows that the group $G$ is not transitive on $\binom{\V}{k}$, 
that is to say, $G$ is not $k$-homogeneous on $\V$. In particular $k\geq 3$, and 
$G$ does not contain the alternating group $A_v$.  We make a few preliminary observations.

\medskip\noindent
\emph{Notation:}\quad  
Let $\ga\in\Ga$ so that $G_\ga$ is transitive on $\ga\times\ov\ga$. 
In particular $k(v-k)$ divides $|G_\ga|$, and $G$ is not $k$-homogeneous. 
Let $G_\ga\leq H<G$ with $H$ maximal in $G$.  If $H$ is intransitive on $\V$ then, 
as $G_\ga$ has only two orbits, we must have $H=G_\ga$. On the other hand, if $H$ is 
transitive on $\V$, we have the following information. 

\begin{lemma}\label{obs}
If $H$ is transitive on $\V$ and $N$ is a normal subgroup of $H$ which is intransitive on $\V$,
then $\ga$ is a union of $N$-orbits.
\end{lemma}

\begin{proof}
Let $\alpha$ be an $N$-orbit containing a point of $\ga$, say $\u$. 
Then $\alpha$ is a block of imprimitivity for $H$ in $\V$. Thus, if $\alpha\subseteq \ga$, 
then $\ga$ is a union of $G_\ga$-translates of $\alpha$ and the result follows. 
Assume then that $\alpha$ also contains a point of $\bar\ga$. Then $G_{\ga,\u}$ fixes 
$\alpha$ set-wise, and also is transitive on $\bar\ga$. Hence $\bar\ga\subset\alpha$. 
This implies that $\ga$ contains all $N$-orbits distinct from $\alpha$, while meeting 
$\alpha$ in a proper non-empty subset. Hence $|\ga|>v/2$, which is 
a contradiction.
\end{proof}

We deal with each of the sporadic almost simple 2-transitive groups in turn. We use 
information from the Atlas~\cite{At}, supplemented in some cases with the aid of the 
computer system {\sf GAP} \cite{GAP} as explained in
Section~\ref{sect:computational}.

We give the proof for each group separately. Since we construct all
codes explicitly on the computer, we can easily check that in the
cases with $k=v/2$ (lines $3$, $16$ and $25$--$27$ in
Table~\ref{tbl-sporad}) the code is self-complementary. We use the notation
introduced at the end of Section~\ref{sect:organisation}.

\subsection{$G=L_2(11)$ with $v=11$} Here $3\leq k\leq 5$, and
since $k(11-k)$ divides $|G_\ga|$, $k$ must be 5, and $G_\ga=A_5$
with orbits of sizes 5, 6 on $\V$. The set $\Ga$ consists of the
blocks of the unique Hadamard $2$-$(11,5,2)$ design (a 2-transitive
biplane) as in Line 1 of Table~\ref{tbl-sporad}, and is $G$-strongly
incidence-transitive. By a result of Ryser (see \cite[Proposition
3.2]{BJL}) each pair of codewords meet in exactly two points of $\V$ and
hence $\delta(\Gamma)=3$.

\subsection{$G=A_7$ with $v=15$ acting on $\PG(3,2)$} Here $3\leq
k\leq 7$, and $k(15-k)$ divides $|G_\ga|$, and hence $k=3$ or $7$.
There are two $G$-orbits on 3-subsets, namely the lines and triangles
of $\PG(3,2)$. If $\ga$ is a line then $G_\ga=A_7\cap(S_3\times
S_4)$ and is transitive on $\ga\times\ov\ga$, as in Line 23 of
Table~\ref{tbl-sporad}, and is $G$-strongly incidence-transitive. On
the other hand the stabiliser $S_3$ of a triangle does not have this
property. If $k=7$, then the orbits of $G_\ga$ in $\V$ have lengths 7
and 8, and it follows that $\ga$ is a plane and $G_\ga=H=L_2(7)$ is
its stabiliser, as in Line 2 of Table~\ref{tbl-sporad}, and $\Ga$ is
$G$-strongly incidence-transitive.

\subsection{$G={\rm M}_{11}$ and $v=11$} Here $5\leq k\leq 11/2$ since
$G$ is $4$-transitive, so
$k=5$. It follows that $G=S_5$ and $\Ga$ is the set of pentads of the
Witt $4$-$(11,5,1)$ design, as in Line 24 of Table~\ref{tbl-sporad}, and
$\Ga$ is $G$-strongly incidence-transitive.

\subsection{$G={\rm M}_{11}$ and $v=12$}\label{m11on12} Here $4\leq
k\leq 6$ since $G$ is $3$-transitive. 
Since $k(12-k)$ divides $|G_\ga|$, it follows that $k=6$
and $|G_\ga|$ is divisible by 36. By \cite[page 18]{At}, the only
maximal subgroups $H$ with order divisible by 36 are $H_1=A_6\cdot2$ and
$H_2=3^2:Q_8.2$, and each of these is transitive on $\V$. Thus $G_\ga$
is a proper subgroup of $H_i$ for some $i$. The group $G$ has two orbits
on the 132 blocks of the Witt $5$-$(12,6,1)$ design, of lengths 22 and
110, and the subgroups $H_1, H_2$ are stabilisers of blocks in these
orbits. Thus $\ga$ is a block of the Witt design.

The first group $H_1$ is imprimitive on $\V$ with 2 blocks of length 6,
and hence if $G_\ga< H_1$ then $G_\ga=H_1'=A_6$ is the stabiliser of
a total, a certain subset of 22 blocks of the Witt design, as in Line
3 of Table~\ref{tbl-sporad}. In particular $G_\ga$ is transitive on
$\ga\times\bar\ga$ so $G$ is strongly incidence-transitive on $\Ga$.
(This code corresponds to a set of 22 words in the ternary Golay code
preserved by $M_{11}$ - see \cite[p.18]{At}).

Now suppose that $G_\ga<H_2$. The group $G_\ga$ contains the normal
Sylow 3-subgroup $P$ of $H_2$ and $P$ has 4 orbits in $\V$ of length 3.
An element $h\in H_2$ of order 8 permutes the $P$-orbits transitively,
so $P< G_\ga\leq 3^2:Q_8$, the 6-subset $\ga$ is a union of two
$P$-orbits, and $H_2$ induces $D_8$ on the set of $P$-orbits. This
$H_2$-action has a unique set of blocks of size 2. If $\ga$ were a
union of one $P$-orbit from each of these blocks then $G_\ga$ would
have order only 18. Hence $\ga$ is the union of $P$-orbits in one
of the $H_2$-blocks and $G_\ga=P:Q_8$. It is not difficult to check
that $G_\ga$ is transitive on $\ga\times\bar\ga$, as in Line 25 of
Table~\ref{tbl-sporad}.

\subsection{$G={\rm M}_{12}$ and $v=12$} Here $6\leq k\leq\frac{v}{2}$ since
$G$ is $5$-transitive, so $k=6$. Suppose first that $H$ is intransitive
on $\V$. Then $H$ has two orbits of length 6, namely $\ga$ and
$\bar\ga$, but by \cite[page 33]{At} there is no such maximal subgroup.
Hence $H$ is transitive on $\V$. Suppose first that $H=A_6\cdot 2^2$,
the stabiliser of a hexad pair. Let $N$ be the index 2 subgroup
fixing the two hexads set-wise. By Lemma~\ref{obs}, $\ga$ is one of
these hexads, that is, a block of the Witt $5$-$(12,6,1)$ design.
Hence $G_\ga=N$ is transitive on $\ga\times\bar\ga$, as in Line 26 of
Table~\ref{tbl-sporad}.

Assume from now on that $\ga$ is not a hexad. 

Suppose next that $H={\rm M}_{11}$ acting transitively on $\V$. Then, by
our arguments in \ref{m11on12} above, $G_\ga$ is the stabiliser in $H$
of a hexad, which contradicts our assumption that $\ga$ is not a hexad.
Then, since $|H|$ is divisible by 36 and $H$ is transitive on $\V$, we
see from \cite[page 33]{At} that the remaining cases are $H_1=M_9:S_3$
stabilising `linked threes' and $H_2=A_4\times S_3$ stabilising a
`$4\times 3$ array'. In the former case, $O_3(H)$ has four orbits of
length 3, any two of which form a hexad (see \cite[p.31]{At}). This
implies that $\ga$ is a hexad, contradiction.
Thus $H=A_4\times S_3$ stabilising a `$4\times 3$ array'. The subgroup
$G_\ga$ must be transitive on the three columns of the array, so $\ga$
must be a union of two of the rows. This however implies that, for
$u\in\ga$, $G_{\ga,u}$ cannot be transitive on $\bar\ga$.

\subsection{$G={\rm M}_{22}$ with $v=22$}
Here $4\leq k\leq 11$ because $G$ is $3$-transitive, and $k\ne 5,9,11$ since
$k(22-k)$ divides $|G_\ga|$. All maximal subgroups
have two orbits on $\V$ (see \cite[page 39]{At}), so $G_\ga = H$. We
obtain examples of subgroups $H$ with orbits of lengths $k$ and $22-k$ as
follows: $2^4:A_6$ with $k=6$ (hexads, that is, certain blocks of the
$5$-$(24,8,1)$ design with two points removed); $A_7$ with $k=7$ (heptads,
that is, blocks of the $5$-$(24,8,1)$ design minus a point;
$2^3:L_3(2)$ with $k=8$ (octads, that is, blocks
of the $5$-$(24,8,1)$ design); $M_{10}\cong A_6 \cdot 2_3$ 
with $k=10$ (decad). There is
no suitable subgroup when $k=4$. In all these cases the group $H$ is
transitive on $\ga\times\ov\ga$, and we have the examples in Lines 4--7
of Table~\ref{tbl-sporad}.

\subsection{$G= \Aut({\rm M}_{22})={\rm M}_{22}.2 $ with $v=22$}
This is very similar to the ${\rm M}_{22}$ case. Again, $4\leq k\leq 11$ 
because $G$ is $3$-transitive, and since
$k(22-k)$ divides $|G_\ga|$, $k\ne 5,9,11$. From the intransitive
maximal subgroups of $G$ we get $3$ more codes which are in Lines
8, 10 and 11 of Table~\ref{tbl-sporad}. They have the same parameters
as the ones for ${\rm M}_{22}$ in Lines 4, 6 and 7 respectively.
The transitive maximal subgroups $L_3(4):2_2$ and $L_2(11):2$ of $G$ 
provide no further example which we verified using the computational
approach described in Section~\ref{sect:computational}. However, 
the maximal chain $A_7 < {\rm M}_{22}< G$ gives a further example
which is in Line 9. The other maximal subgroups of ${\rm M}_{22}$ do
not provide any new code for $G$ since they are properly contained in
intransitive maximal subgroups of $G$.

\subsection{$G={\rm M}_{23}$ and $v=23$} Here $5\leq k\leq 11$ because $G$
is $3$-transitive, and $k \ne 6, 10$
since $k(23-k)$ divides $|G_\ga|$. All maximal subgroups
have two orbits in $\V$, except the last-listed subgroup $23:11$ on
\cite[page 71]{At}, which is too small to have $G_\ga$ as a subgroup,
Hence $G_\ga=H$. We obtain examples of subgroups with orbits of lengths
$k$ and $23-k$ as follows: $H=2^4:A_7$ with $k=7$, $\ga$ a block of
the $4$-$(23,7,1)$ Witt design; $H=A_8$ with $k=8$, blocks $\ga$ being
certain octads of the $5$-$(24,8,1)$ design; and $H={\rm M}_{11}$ with
$k=11$, $\ga$ an endecad (a dodecad of the $5$-$(24,8,1)$ design minus a
point). Each of these subgroups $H$ has three orbits on pairs of points
(see \cite[p.71]{At}) and hence is transitive on $\ga\times\ov\ga$, so
we have the examples in Lines 12--14 of Table~\ref{tbl-sporad}.

\subsection{$G={\rm M}_{24}$ and $v=24$} Here $6\leq k\leq12$ because $G$
is $5$-transitive, and $k \ne 7, 11$
since $k(24-k)$ divides $|G_\ga|$. The last two maximal
subgroups $L_2(23), L_2(7)$ listed on \cite[p.96]{At} have orders not
divisible by $k(24-k)$ for any suitable $k$. Four of the remaining
maximal subgroups $H$ are intransitive on $\V$, and of these only
$H=2^4:A_8$ has shortest orbit length at least $6$ and that length
is $k=8$ with $\ga$ an octad, that is, a block of the $5$-$(24,8,1)$
design. In this case $G_\ga=H$ has orbits in $\V$ of lengths 8, 16,
and the stabiliser of a point $u\in\ga$, namely $2^4:A_7$, is a
maximal subgroup of $M_{23}$ with two orbits in $\V\setminus\{u\}$
by \cite[p.71]{At}. Thus $G_\ga$ is transitive on $\ga\times\bar\ga$
and we have the example in Line 15 of Table~\ref{tbl-sporad}. 

So we may assume that $H$ is one of the three transitive maximal
subgroups ${\rm M}_{12}:2$ or $2^6:3\dot{}\,S_6$ or $2^6:(L_3(2)\times
S_3)$, stabilising a `duum', `sextet', or `trio' respectively (see
\cite[pp.94--96]{At}). In the first case, $\ga$ must be one of the two
orbits of $H'={\rm M}_{12}$ of length $12$, by Lemma~\ref{obs}. Hence
$G_\ga=H'$, which is transitive on $\ga\times\bar\ga$, as in Line 16 of
Table~\ref{tbl-sporad}. In the second case, the `sextet' is a partition
of $\V$ into six `tetrads' any two of which form an octad, and any
division into two sets of three tetrads forms a duum.
This case was treated computationally: examining each of the maximal 
subgroups of $H$ lead to exactly one example, arising from the subgroup 
$2^6:3.(S_3 \times S_3)$, as in Line 27 of Table~\ref{tbl-sporad}. 

Finally consider the third
case $H=2^6:(L_3(2)\times S_3)$. 
Let $N=2^6:L_3(2)$, a normal subgroup of $H$ with three orbits of
length 8, each of them an octad. By Lemma~\ref{obs}, $\ga$ is one of
these $N$-orbits, which is a contradiction.

\subsection{$G=\PGaL_2(8)=L_2(8).3=\Ree(3)$ with $v=28$}
\label{PSL28}
Our computational approach described in Section~\ref{sect:computational}
readily proves that this group does not provide an example. The same is 
true for all the Ree groups and is proved theoretically 
in \cite[Proposition 9.3]{LP}.

\subsection{$G={\rm HS}$ with $v=176$} Here $3\leq k\leq 88$. Suppose
first that $G_\ga$ is contained in a proper transitive subgroup $H$ of
$G$. Then, for $u\in\ga$, $G=HG_u$ and it follows from \cite[Section
6.7]{factns} that $H={\rm M}_{22}$ and $G_{\ga,\u}\leq H_\u=A_7$.
Since $G_\ga$ has two orbits on $\V$ it follows from the permutation
characters given in \cite[page 39]{At} that $G_\ga=L_3(4)$ or
$2^4:A_6$.
In the former case, $|H_\u:G_{\u,\ga}|=$ 7 or 15, and                   
$|G_\ga:G_{\ga,\u}|=$ 56 or 120 respectively. Thus $k=56$ and for       
$\u\in\ga$, $G_{\ga,\u}=A_6$; however $G_{\ga,\u}$ has at least 2       
orbits in $\ov\ga$ (see the permutation characters given in \cite[page  
23]{At}). Thus $G_\ga=2^4:A_6$. In this case, $|H_\u:G_{\u,\ga}|=$ 35   
or 42, and $|G_\ga:G_{\ga,\u}|=$ 80 or 96 respectively. Thus $k=80$,    
but then $k(v-k)$ does not divide $|2^4:A_6|$.                          

Thus we conclude that $G_\ga=H$ is a maximal intransitive subgroup of
$G$ with two orbits in $\V$, and since $k\geq3$, $H$ is not a vertex or
edge stabiliser of the Higman Sims graph. If $G_\ga\cong U_3(5):2$ has
index 176 in $G$ then it is the stabiliser of a `quadric' and $k=50$
(see \cite[page 80]{At}). Moreover (see \cite[page 34]{At}) $G_\ga$ is
indeed transitive on $\ga\times\ov\ga$, and we have the example in Line
17 of Table~\ref{tbl-sporad}.
This is the $2$-$(176,50,14)$ design constructed in \cite{Higman}.

If $G_\ga \cong L_3(4):2_1$ we get another example with $k=56$ because
$G_\ga$ is transitive on $\ga \times \ov\ga$. This example is given in 
Line~18 of Table~\ref{tbl-sporad}. All other intransitive maximal subgroups
have either more than $2$ orbits or are not transitive on $\ga \times
\ov\ga$, so do not give rise to any more examples. We established these facts
using direct computations in \textsf{GAP}.

\subsection{$G={\rm Co}_3$ with $v=276$} Here $3\leq k\leq 138$. The
information in \cite[p.134]{At} is not sufficient to deal with this
group so we use the {\sf GAP} computer system as described in
Section~\ref{sect:computational}. There are 14 conjugacy
classes of maximal subgroups $H$. All are intransitive on $\V$, so
$G_\ga = H$. A {\sf GAP} computation showed that exactly eight of
the classes have two orbits on $\V$ --- these orbits will be $\ga$ and
$\bar\ga$. For a representative $H$ of each of these eight classes, we
chose a point $u$ in one of the $H$-orbits and computed the number of
$H_u$-orbits in the other $H$-orbit. This number is 1 if and only if
$H$ is transitive on $\ga\times\bar\ga$. The computed number of orbits
was 1 for exactly five maximal subgroups. One of these was the point
stabiliser $McL:2$ which does not lead to an example since in that case
we would have $k=1$. Thus we obtain exactly four examples as in Lines
19-22 of Table~\ref{tbl-sporad}.

Note that the code in Line 21 is the $2$-$(276,100,1458)$ design
presented in \cite{Haemers}, and that the other three are also
$2$-$(276,k,\lambda)$ designs for certain values of $\lambda$. We
established the latter fact using the \textsf{DESIGN} \textsf{GAP} package \cite{DESIGNGAP}.

\end{document}